\documentclass[11pt]{article}
\usepackage{amsmath}
\usepackage{amssymb}
\usepackage{amsthm}
\usepackage{bookmark}
\usepackage{geometry}
\usepackage{hyperref}
\newtheorem{theorem}{Theorem}
\newtheorem{corollary}[theorem]{Corollary}
\newtheorem{lemma}[theorem]{Lemma}
\newtheorem{proposition}[theorem]{Proposition}
\theoremstyle{definition}
\newtheorem{example}[theorem]{Example}
\DeclareMathOperator{\im}{im}
\DeclareMathOperator{\lcm}{lcm}
\DeclareMathOperator{\rad}{rad}
\DeclareMathOperator{\rank}{rank}
\renewcommand{\labelenumi}{\upshape(\roman{enumi})}
\renewcommand{\theenumi}\labelenumi

\title{On sesquilinear forms over finite fields}
\author{Ruikai Chen\\\small School of Mathematical Sciences, South China Normal University, Guangzhou 510631, China\\\small Email: \href{mailto:chen.rk@outlook.com}{chen.rk@outlook.com}}
\date{}

\begin{document}

\maketitle

\begin{abstract} 
We develop a theory of sesquilinear forms over finite fields, investigating their representations via polynomials and coefficient matrices, along with classification results for these forms. Through their connection to quadratic forms, we calculate certain character sums to resolve enumeration problems for equations defined by sesquilinear forms. This provides a characterization of a class of maximal or minimal Artin-Schreier curves with explicit examples.

\end{abstract}

\section{Introduction}

Let $K$ be a field equipped with an automorphism $\pi$ and $V$ be a vector space over $K$. A linear form of $V/K$ is a $K$-linear map from $V$ to $K$, and a semilinear form of $V/K$ is a $K$-semilinear map (with respect to $\pi$) from $V$ to $K$. A sesquilinear form $\sigma$ of $V/K$ is then defined to be a map from $V\times V$ to $K$ such that $\sigma(\cdot,v)$ is a linear form and $\sigma(v,\cdot)$ is a semilinear form of $V/K$ for each fixed $v\in V$. Denote by $\sigma^*$ the sesquilinear form satisfying $\sigma^*(u,v)=\pi(\sigma(v,u))$ for all $u,v\in V$. When $\pi$ is the trivial automorphism, $\sigma$ is just a bilinear form. For our purpose, $\pi$ is always assumed to be a nontrivial involution of $K$, so that $\sigma^{**}=\sigma$. If $\sigma^*=\lambda\sigma$ for some $\lambda\in K$, then $\sigma$ is called a $\lambda$-Hermitian form of $V/K$. The left radical of $\sigma$ is the subspace of $V/K$ consisting of $u\in V$ such that $\sigma(u,v)=0$ for all $v\in V$, denoted by $\rad\sigma$, and the right radical is defined to be $\rad\sigma^*$. If $\sigma(v,v)=0$ for all $v\in V$, then $\sigma$ is called alternating. Two sesquilinear forms $\sigma_0$ and $\sigma_1$ of $V/K$ are called equivalent, if there exists a linear automorphism $\tau$ of $V/K$ such that $\sigma_0(u,v)=\sigma_1(\tau(u),\tau(v))$ for all $u,v\in V$. 

In general, if $V/K$ is finite-dimensional, then sesquilinear forms of $V/K$ are represented by matrices. Given a basis $\beta_1,\dots,\beta_n$ of $V/K$, for $u=\sum_{i=1}^nu_i\beta_i$ and $v=\sum_{i=1}^nv_i\beta_i$ with $u_i,v_i\in\mathbb F_{q^2}$, one has
\[\sigma(u,v)=\sum_{i=1}^n\sum_{j=1}^n\sigma(u_i\beta_i,v_j\beta_j)=\sum_{i=1}^n\sum_{j=1}^nu_i\pi(v_j)\sigma(\beta_i,\beta_j).\]
The $n\times n$ matrix with $(i,j)$ entry $\sigma(\beta_i,\beta_j)$ is thus called the coefficient matrix of $\sigma$ with respect to the basis $\beta_1,\dots,\beta_n$.

We focus on sesquilinear forms over finite fields. Let $K=\mathbb F_{q^2}$, the finite field that admits an involutive automorphism $\pi$ raising each element to its $q$-th power, and let $V=\mathbb F_{q^{2n}}$, as an $n$-dimensional vector space over $\mathbb F_{q^2}$. A distinguishing feature of sesquilinear forms over finite fields, compared to those over general fields, is that they can be represented by polynomials. Let $L$ be a $q^2$-linear polynomial over $\mathbb F_{q^{2n}}$, where the exponent of each term is a power of $q^2$, and let $\mathrm{Tr}$ be the trace map of $\mathbb F_{q^{2n}}/\mathbb F_{q^2}$. Define $\sigma_L$ by
\[\sigma_L(u,v)=\mathrm{Tr}(uL(v^q))\]
for $u,v\in\mathbb F_{q^{2n}}$. This clearly induces a sesquilinear form of $\mathbb F_{q^{2n}}/\mathbb F_{q^2}$.

In the literature, sesquilinear forms over finite fields have received far less attention than symmetric bilinear and quadratic forms. For example, \cite{klapper1997cross} and \cite{fitzgerald2017invariants} classify quadratic forms represented by monomials, while \cite{coulter2002number} studies the number of zeros of such forms. In recent years, quadratic forms inducing maximal or minimal Artin-Schreier curves have attracted significant interest, which can be found in \cite{anbar2014quadratic,ozbudak2016explicit,bartoli2021explicit,oliveira2023artin}. This paper addresses analogous problems for sesquilinear forms. We will first represent sesquilinear forms by polynomials over finite fields and study their classification in a general framework. Afterwards, we investigate equations defined by sesquilinear forms, leading to a characterization of maximal or minimal Artin-Schreier curves derived from these forms.

\section{Sesquilinear forms represented by polynomials and their classification}

Consider the polynomial ring $\mathbb F_{q^{2n}}[x]$, where polynomials viewed as maps on $\mathbb F_{q^{2n}}$ are usually taken modulo $x^{q^{2n}}-x$. A $q^2$-linear polynomial over $\mathbb F_{q^{2n}}$ thereby acts as a linear endomorphism of $\mathbb F_{q^{2n}}/\mathbb F_{q^2}$. Conversely, every linear endomorphism of $\mathbb F_{q^{2n}}/\mathbb F_{q^2}$ can be written as a $q^2$-linear polynomial over $\mathbb F_{q^{2n}}$. In particular, every linear form of $\mathbb F_{q^{2n}}/\mathbb F_{q^2}$ is given by $\mathrm{Tr}(\alpha x)$ for some $\alpha\in\mathbb F_{q^{2n}}$. This correspondence yields canonical polynomial representation for sesquilinear forms of $\mathbb F_{q^{2n}}/\mathbb F_{q^2}$.

\begin{theorem}\label{rep}
Every sesquilinear form of $\mathbb F_{q^{2n}}/\mathbb F_{q^2}$ can be written as $\sigma_L$ for some $q^2$-linear polynomial $L$ over $\mathbb F_{q^{2n}}$. In addition, the following conditions are equivalent:
\begin{enumerate}
\item\label{item0}$\sigma_L$ is alternating;
\item\label{item1}$\sigma_L$ is the zero form;
\item\label{item2}$L(x)\equiv0\pmod{x^{q^{2n}}-x}$.
\end{enumerate}
\end{theorem}
\begin{proof}
Let $\sigma$ be a sesquilinear form of $\mathbb F_{q^{2n}}/\mathbb F_{q^2}$. For every $v\in\mathbb F_{q^{2n}}$, the polynomial $\sigma(x,v)$ in $\mathbb F_{q^{2n}}[x]$ induces a linear form of $\mathbb F_{q^{2n}}/\mathbb F_{q^2}$, so there exists a map $\varphi$ on $\mathbb F_{q^{2n}}$ such that $\sigma(u,v)=\mathrm{Tr}(u\varphi(v))$ for all $u,v\in\mathbb F_{q^{2n}}$. Take a basis $\beta_1,\dots,\beta_n$ of $\mathbb F_{q^{2n}}/\mathbb F_{q^2}$ and a $q^2$-linear polynomial $L$ such that $L(\beta_i^q)=\varphi(\beta_i)$ for $1\le i\le n$. For $v=\sum_{i=1}^nv_i\beta_i$ with $v_i\in\mathbb F_{q^2}$, we have
\[\sigma(u,v)=\sum_{i=1}^nv_i^q\sigma(u,\beta_i)=\sum_{i=1}^nv_i^q\mathrm{Tr}(u\varphi(\beta_i))=\sum_{i=1}^nv_i^q\mathrm{Tr}(uL(\beta_i^q))=\mathrm{Tr}(uL(v^q)),\]
as asserted.

If $\sigma$ is alternating, then, given $u,v\in\mathbb F_{q^{2n}}$,
\[c^q\sigma(u,v)+c\sigma(v,u)=\sigma(u,u)+\sigma(u,cv)+\sigma(cv,u)+\sigma(cv,cv)=\sigma(u+cv,u+cv)=0\]
for all $c\in\mathbb F_{q^2}$; this forces $\sigma(u,v)=\sigma(v,u)=0$. Since $u$ and $v$ are arbitrary from $\mathbb F_{q^{2n}}$, it follows that \ref{item0} implies \ref{item1}.

To see that \ref{item1} implies \ref{item2}, assume $\mathrm{Tr}(uL(v^q))=0$ for all $u,v\in\mathbb F_{q^{2n}}$. Then $L(v^q)$ lies in the kernel of every linear form of $\mathbb F_{q^{2n}}/\mathbb F_{q^2}$. Hence, $L(v^q)$ is necessarily zero for any $v\in\mathbb F_{q^{2n}}$ and $L(x)\equiv0\pmod{x^{q^{2n}}-x}$. The rest is obvious.
\end{proof}

We establish a connection between the polynomial representing a sesquilinear forms and the coefficient matrices. For $L(x)=\sum_{i=0}^{n-1}a_ix^{q^{2i}}$, let
\[M_L=\begin{pmatrix}a_0&a_1&\cdots&a_{n-1}\\a_{n-1}^{q^2}&a_0^{q^2}&\cdots&a_{n-2}^{q^2}\\\vdots&\vdots&\ddots&\vdots\\a_1^{q^{2(n-1)}}&a_2^{q^{2(n-1)}}&\cdots&a_0^{q^{2(n-1)}}\end{pmatrix}.\]
For a matrix $M$ over $\mathbb F_{q^{2n}}$, denote by $M^*$ the matrix obtained by raising each entry of the transpose of $M$ to its $q$-th power. Then $M^{**}=M$ if and only if $M$ has entries all in $\mathbb F_{q^2}$.

Given a basis $\beta_1,\dots,\beta_n$ of $\mathbb F_{q^{2n}}/\mathbb F_{q^2}$, the coefficient matrix of $\sigma_L$ with respect to the basis arises from the matrix $M_L$ defined above. The basis is associated with a nonsingular matrix
\begin{equation}\label{B}B=\begin{pmatrix}\beta_1&\beta_1^{q^2}&\cdots&\beta_1^{q^{2(n-1)}}\\\beta_2&\beta_2^{q^2}&\cdots&\beta_2^{q^{2(n-1)}}\\\vdots&\vdots&\ddots&\vdots\\\beta_n&\beta_n^{q^2}&\cdots&\beta_n^{q^{2(n-1)}}\end{pmatrix},\end{equation}
so that
\[M_LB^*=\begin{pmatrix}L(\beta_1^q)&L(\beta_2^q)&\cdots&L(\beta_n^q)\\L(\beta_1^q)^{q^2}&L(\beta_2^q)^{q^2}&\cdots&L(\beta_n^q)^{q^2}\\\vdots&\vdots&\ddots&\vdots\\L(\beta_1^q)^{q^{2(n-1)}}&L(\beta_2^q)^{q^{2(n-1)}}&\cdots&L(\beta_n^q)^{q^{2(n-1)}}\end{pmatrix}.\]
The rank of this matrix is equal to $n-\dim\ker L$, the dimension of the subspace of $\mathbb F_{q^{2n}}/\mathbb F_{q^2}$ spanned by $L(\beta_1^q),\dots,L(\beta_n^q)$. The $(i,j)$ entry of $BM_LB^*$ is $\mathrm{Tr}(\beta_iL(\beta_j^q))$, so $BM_LB^*$ is a coefficient matrix of $\sigma_L$ of rank $n-\dim\ker L$.

Via a linear automorphism $\tau$ of $\mathbb F_{q^{2n}}/\mathbb F_{q^2}$, an equivalent form of $\sigma_L$ has a coefficient matrix with $(i,j)$ entry $\sigma_L(\tau(\beta_i),\tau(\beta_j))$. This resulting matrix is the coefficient matrix of $\sigma_L$ with respect to another basis $\tau(\beta_1),\dots,\tau(\beta_n)$ of $\mathbb F_{q^{2n}}/\mathbb F_{q^2}$. If $C$ is a coefficient matrix of $\sigma_L$, then all its coefficient matrices are congruent to $C$ by change of basis; here, a matrix is said to be congruent to $C$ if it can be written as $TCT^*$ for some $n\times n$ nonsingular matrix $T$ over $\mathbb F_{q^2}$. Therefore, equivalent sesquilinear forms have congruent coefficient matrices, and vice versa.

For $L(x)=\sum_{i=0}^{n-1}a_ix^{q^{2i}}$, let
\[L^*(x)=\sum_{i=0}^{n-1}(a_i^qx)^{q^{2(n-i-1)}}.\]
This polynomial is defined so that $\sigma_L^*=\sigma_{L^*}$, as shown below.

\begin{proposition}
The polynomial $L^*$ is the unique $q^2$-linear polynomial over $\mathbb F_{q^{2n}}$ of degree less than $q^{2n}$ such that $\mathrm{Tr}(uL^*(v^q))=\mathrm{Tr}(vL(u^q))^q$ for all $u,v\in\mathbb F_{q^{2n}}$. If $C$ is the coefficient matrix of $\sigma_L$ with respect to a basis of $\mathbb F_{q^{2n}}/\mathbb F_{q^2}$, then $C^*$ is the coefficient matrix of $\sigma_L^*$ with respect to the same basis.
\end{proposition}
\begin{proof}
A direct calculation yields
\[\mathrm{Tr}(uL^*(v^q))=\sum_{i=0}^{n-1}\mathrm{Tr}\big(u(a_i^qv^q)^{q^{2(n-i-1)}}\big)=\sum_{i=0}^{n-1}\mathrm{Tr}\big(a_ivu^{q^{2i+1}}\big)^q=\mathrm{Tr}(vL(u^q))^q.\]
The uniqueness follows from Theorem \ref{rep}, and the second statement is an immediate consequence.
\end{proof}

As a result, the following statements can be easily checked:
\begin{itemize}
\item $L^{**}=L$;
\item $(\alpha L)^*=\alpha^qL^*$ for arbitrary $\alpha\in\mathbb F_{q^2}$;
\item $\rad\sigma_L$ is the orthogonal complement of $\im L$ with respect to the standard bilinear form of $\mathbb F_{q^{2n}}/\mathbb F_{q^2}$ induced by $\mathrm{Tr}$;
\item $\rad\sigma_L^*=\pi^{-1}(\ker L)$, where $\pi$ is the automorphism that raises an element to its $q$-th power;
\item $\dim\rad\sigma_L=\dim\rad\sigma_L^*$;
\item $M_{L^*}=PM_L^*$, where $P$ is the $n\times n$ permutation matrix
\begin{equation}\label{P}\begin{pmatrix}0&\cdots&0&1\\1&\cdots&0&0\\\vdots&\ddots&\vdots&\vdots\\0&\cdots&1&0\end{pmatrix}.\end{equation}
\end{itemize}

We are now ready for the classification of sesquilinear forms over finite fields.

\begin{lemma}[{\cite[Theorem 2.3.1]{wall1963conjugacy}}]
For a matrix $A$ in $\mathrm{GL}_n(q^2)$, the general linear group formed by $n\times n$ nonsingular matrices over $\mathbb F_{q^2}$, $A^*$ is conjugate to $A^{-1}$ in $\mathrm{GL}_n(q^2)$ if and only if $A=C^*C^{-1}$ for some $C\in\mathrm{GL}_n(q^2)$.
\end{lemma}

\begin{theorem}\label{equiv}
Let $L_0$ and $L_1$ be $q^2$-linear polynomials over $\mathbb F_{q^{2n}}$ of degree less than $q^{2n}$, $C_0$ and $C_1$ be some coefficient matrices of $\sigma_{L_0}$ and $\sigma_{L_1}$ respectively, and suppose that they are both nonsingular. Then the following conditions are equivalent:
\begin{enumerate}
\item\label{item3}$\sigma_0$ and $\sigma_1$ are equivalent;
\item\label{item4}$C_0^*C_0^{-1}$ and $C_1^*C_1^{-1}$ are conjugate in $\mathrm{GL}_n(q^2)$;
\item $PM_{L_0}^*M_{L_0}^{-1}$ and $PM_{L_1}^*M_{L_1}^{-1}$ are conjugate in $\mathrm{GL}_n(q^{2n})$, where $P$ is the permutation matrix as in \eqref{P}.
\end{enumerate}
\end{theorem}
\begin{proof}
Let $\varphi$ be the map on $\mathrm{GL}_n(q^2)$ defined by $\varphi(A)=A^*A^{-1}$ for $A\in\mathrm{GL}_n(q^2)$. Then $\im\varphi$ consists precisely of $A$ in $\mathrm{GL}_n(q^2)$ such that $A^*$ is conjugate to $A^{-1}$, according to the previous lemma. Consider the action of $\mathrm{GL}_n(q^2)$ on itself by congruence, where each $T\in\mathrm{GL}_n(q^2)$ acts on $\mathrm{GL}_n(q^2)$ by $A\mapsto TAT^*$. If there exists $A\in\mathrm{GL}_n(q^2)$ that is fixed by $T$, then $T\in\im\varphi$, as
\[T^*=(TA)^{-1}TAT^*=(TA)^{-1}A=A^{-1}T^{-1}A.\]
Assume $T\in\im\varphi$, so that $T^*=A_0^{-1}T^{-1}A_0$ for some $A_0\in\mathrm{GL}_n(q^2)$. In this case, for $A\in\mathrm{GL}_n(q^2)$, $A$ is fixed by $T$ if and only if $AA_0^{-1}$ belongs to the centralizer of $T$. In fact,
\[TAT^*=TAA_0^{-1}(AA_0^{-1}T)^{-1}AA_0^{-1}TA_0T^*=TAA_0^{-1}(AA_0^{-1}T)^{-1}A.\]
Consequently, the number of fixed points of $T$ in $\mathrm{GL}_n(q^2)$ is the order of the centralizer of $T$, denoted by $\nu(T)$, if $T\in\im\varphi$, and is $0$ otherwise. By Burnside's lemma, the number of orbits of the action is thereby
\[\frac1{|\mathrm{GL}_n(q^2)|}\sum_{T\in\im\varphi}\nu(T).\]
Moreover, for $T\in\im\varphi$ with $A\in\mathrm{GL}_n(q^2)$ fixed by $T$, $STS^{-1}\in\im\varphi$ for every $S\in\mathrm{GL}_n(q^2)$, since
\[(STS^{-1})SAS^*(STS^{-1})^*=STAS^*(S^{-1})^*T^*S^*=STAT^*S^*=SAS^*.\]
Then $\mathrm{GL}_n(q^2)$ acts on $\im\varphi$ by conjugation, and the stabilizer of an element of $\im\varphi$ is exactly its centralizer. The number of conjugacy classes in $\im\varphi$ is also
\[\frac1{|\mathrm{GL}_n(q^2)|}\sum_{T\in\im\varphi}\nu(T).\]

Suppose that $C_0=TC_1T^*$ for some $T\in\mathrm{GL}_n(q^2)$. Then
\[C_0^*C_0^{-1}=TC_1^*T^*(T^*)^{-1}C_1^{-1}T^{-1}=TC_1^*C_1^{-1}T^{-1},\]
which means $C_0^*C_0^{-1}$ is conjugate to $C_1^*C_1^{-1}$. Hence, given the above arguments on the numbers of orbits, the map $\varphi$ induces a one-to-one correspondence between the congruence classes in $\mathrm{GL}_n(q^2)$ and the conjugacy classes in $\im\varphi$. This establishes the equivalence between \ref{item3} and \ref{item4}.

Let $\beta_1,\dots,\beta_n$ be a basis of $\mathbb F_{q^{2n}}/\mathbb F_{q^2}$ with associated matrix $B$ as in \eqref{B}, so that $\sigma_{L_0}$ has a coefficient matrix $BM_{L_0}B^*$. Note that $B^{**}=BP$ and
\[(BM_{L_0}B^*)^*(BM_{L_0}B^*)^{-1}=B^{**}M_{L_0}^*B^*(B^*)^{-1}M_{L_0}^{-1}B^{-1}=BPM_{L_0}^*M_{L_0}^{-1}B^{-1}.\]
Then $C_0^*C_0^{-1}$ and $PM_{L_0}^*M_{L_0}^{-1}$ are conjugate in $\mathrm{GL}_n(q^{2n})$, as well as $C_1^*C_1^{-1}$ and $PM_{L_1}^*M_{L_1}^{-1}$. Since, for matrices over $\mathbb F_{q^2}$, conjugacy in $\mathrm{GL}_n(q^2)$ is equivalent to conjugacy in $\mathrm{GL}_n(q^{2n})$, the proof is complete.
\end{proof}

\begin{example}\label{2d}
Let $n=2$ and $L(x)=ax^{q^2}+bx$ over $\mathbb F_{q^4}$ with $\mathrm N(a)\ne\mathrm N(b)$, where $\mathrm N$ denotes the norm map of $\mathbb F_{q^4}/\mathbb F_{q^2}$. We classify all sesquilinear forms in this case, by determining canonical forms of their nonsingular coefficient matrices. It can be easily verified that
\[M_L^{-1}=\frac1{\mathrm N(a)-\mathrm N(b)}\begin{pmatrix}-b^{q^2}&a\\a^{q^2}&-b\end{pmatrix}.\]
Consider
\[\begin{split}M&=\begin{pmatrix}0&1\\1&0\end{pmatrix}M_L^*M_L^{-1}\\&=\frac1{\mathrm N(a)-\mathrm N(b)}\begin{pmatrix}0&1\\1&0\end{pmatrix}\begin{pmatrix}b^q&a^{q^3}\\a^q&b^{q^3}\end{pmatrix}\begin{pmatrix}-b^{q^2}&a\\a^{q^2}&-b\end{pmatrix}\\&=\frac1{\mathrm N(a)-\mathrm N(b)}\begin{pmatrix}a^{q^2}b^{q^3}-a^qb^{q^2}&a^{q+1}-b^{q^3+1}\\a^{q^3+q^2}-b^{q^2+q}&ab^q-a^{q^3}b\end{pmatrix},\end{split}\]
where
\[\begin{split}&\mathrel{\phantom{=}}\mathrm N\big(a^{q^2}b^{q^3}-a^qb^{q^2}\big)-\mathrm N\big(a^{q+1}-b^{q^3+1}\big)\\&=\mathrm N(ab^q)^{q^2}+\mathrm N(ab^q)^q-\mathrm{Tr}((ab^q)^{q+1})-\mathrm N(a)^{q+1}-\mathrm N(b)^{q^3+1}+\mathrm{Tr}\big((ab^q)^{q^3+q^2}\big)\\&=\mathrm N(ab^q)+\mathrm N(ab^q)^q-\mathrm N(a)^{q+1}-\mathrm N(b)^{q+1}\\&=-(\mathrm N(a)-\mathrm N(b))^{q+1}.\end{split}\]
Let $\alpha=\mathrm{Tr}(ab^q)$ and $\beta=\mathrm N(a)-\mathrm N(b)$, so that the matrix $M$ has trace $(\alpha-\alpha^q)\beta^{-1}$ and determinant $-\beta^{q-1}$. Observe that the characteristic polynomial of $M$ has discriminant
\[(\alpha-\alpha^q)^2\beta^{-2}+4\beta^{q-1}.\]
If the polynomial has distinct roots in its splitting field, then the conjugacy class of $M$ in $\mathrm{GL}_2(q^4)$ is uniquely determined by $(\alpha-\alpha^q)\beta^{-1}$ and $-\beta^{q-1}$; otherwise, $M$ is either a scalar matrix or conjugate to a nontrivial Jordan block. If $a^{q+1}-b^{q^3+1}=0$, then $M$ is a scalar matrix, since $b\ne0$ and
\[b\big(a^{q^2}b^{q^3}-a^qb^{q^2}\big)=a^{q^2+q+1}-a^qb^{q^2+1}=a^q\beta,\]
with
\[b\big(ab^q-a^{q^3}b\big)=ab^{(q^3+1)q}-a^{q^3}b^2=a^{q^2+q+1}-a^qb^{q^2+1}\big(a^{q+1}b^{-1-q^3}\big)^{q^2-q}=a^q\beta.\]
We conclude that $M$ is conjugate to
\begin{itemize}
\item any $2\times2$ matrix over $\mathbb F_{q^2}$ with trace $(\alpha-\alpha^q)\beta^{-1}$ and determinant $-\beta^{q-1}$ if $(\alpha-\alpha^q)^2+4\beta^{q+1}\ne0$;
\item $\begin{pmatrix}a^qb^{-1}&0\\0&a^qb^{-1}\end{pmatrix}$ if $a^{q+1}-b^{q^3+1}=0$;
\item $\begin{pmatrix}\frac12(\alpha-\alpha^q)\beta^{-1}&0\\1&\frac12(\alpha-\alpha^q)\beta^{-1}\end{pmatrix}$ if $(\alpha-\alpha^q)^2+4\beta^{q+1}=0$ and $a^{q+1}-b^{q^3+1}\ne0$ with $q$ odd;
\item $\begin{pmatrix}\sqrt{\beta^{q-1}}&0\\1&\sqrt{\beta^{q-1}}\end{pmatrix}$ if $(\alpha-\alpha^q)^2+4\beta^{q+1}=0$ and $a^{q+1}-b^{q^3+1}\ne0$ with $q$ even.
\end{itemize}
\end{example}

Next, consider a sesquilinear form $\sigma$ with singular coefficient matrices. If $\mathbb F_{q^{2n}}$ is the direct sum of some subspaces $W$ and $W_0$ of $\mathbb F_{q^{2n}}/\mathbb F_{q^2}$, such that $\sigma(w,v)=\sigma(v,w)=0$ for all $w\in W_0$ and $v\in\mathbb F_{q^{2n}}$, then it suffices to investigate $\sigma$ on $W$. By definition, every such subspace $W_0$ is contained in $\rad\sigma\cap\rad\sigma^*$. In the case $W_0=\rad\sigma\cap\rad\sigma^*$, we call the restriction of $\sigma$ to $W$ a reduced form of $\sigma$.

\begin{proposition}
Two sesquilinear forms $\sigma_0$ and $\sigma_1$ of $\mathbb F_{q^{2n}}/\mathbb F_{q^2}$ are equivalent if and only if $\dim(\rad\sigma_0\cap\rad\sigma_0^*)=\dim(\rad\sigma_1\cap\rad\sigma_1^*)$ and $\sigma_0$ and $\sigma_1$ have reduced forms whose coefficient matrices are congruent.
\end{proposition}
\begin{proof}
Take a basis of $\rad\sigma_0\cap\rad\sigma_0^*$ and extend it to a basis of $\mathbb F_{q^{2n}}/\mathbb F_{q^2}$. With this basis $\sigma_0$ has a coefficient matrix $\begin{pmatrix}C_0&\mathbf0\\\mathbf0&\mathbf0\end{pmatrix}$, where $C_0$ is a coefficient matrix of a reduced form of $\sigma_0$. In the same manner, $\sigma_1$ has a coefficient matrix $\begin{pmatrix}C_1&\mathbf0\\\mathbf0&\mathbf0\end{pmatrix}$.

If $\dim(\rad\sigma_0\cap\rad\sigma_0^*)=\dim(\rad\sigma_1\cap\rad\sigma_1^*)$ and $C_0$ and $C_1$ are congruent, then the coefficient matrices of $\sigma_0$ and $\sigma_1$ are congruent by a nonsingular block diagonal matrix. Conversely, assume that there is a linear automorphism $\tau$ of $\mathbb F_{q^{2n}}/\mathbb F_{q^2}$ such that $\sigma_0(u,v)=\sigma_1(\tau(u),\tau(v))$ for all $u,v\in\mathbb F_{q^{2n}}$. Then $\tau$ maps $\rad\sigma_0\cap\rad\sigma_0^*$ onto $\rad\sigma_1\cap\rad\sigma_1^*$, and thus these two subspaces have the same dimension. Now $C_0$ and $C_1$ have the same size, and $\begin{pmatrix}C_0&\mathbf0\\\mathbf0&\mathbf0\end{pmatrix}$ being congruent to $\begin{pmatrix}C_1&\mathbf0\\\mathbf0&\mathbf0\end{pmatrix}$ implies that $C_0$ is congruent to $C_1$, as easily seen.
\end{proof}

\begin{corollary}\label{reduced}
Let $C$ be a coefficient matrix of a reduced form of a sesquilinear form $\sigma$ of $\mathbb F_{q^{2n}}/\mathbb F_{q^2}$. Then $\sigma$ has a diagonal coefficient matrix whose nonzero entries are $d_1,\dots,d_r$, if and only if $\rad\sigma=\rad\sigma^*$ and $C^*C^{-1}$ is conjugate to the diagonal matrix with entries $d_1^{q-1},\dots,d_r^{q-1}$.
\end{corollary}
\begin{proof}
Note first that the condition $\rad\sigma=\rad\sigma^*$ implies that $C$ is nonsingular, since then $\rank C=n-\dim\rad\sigma=n-\dim(\rad\sigma\cap\rad\sigma^*)$, which is the order of $C$. Suppose that $\sigma$ has a diagonal coefficient matrix whose nonzero entries are $d_1,\dots,d_r$ with $r=\rank C$. Then for some basis $\beta_1,\dots,\beta_n$ of $\mathbb F_{q^{2n}}/\mathbb F_{q^2}$, one has $\sigma(u,v)=\sum_{i=1}^rd_iu_iv_i^q$ if $u=\sum_{i=1}^nu_i\beta_i$ and $v=\sum_{i=1}^nv_i\beta_i$ with $u_i,v_i\in\mathbb F_{q^2}$. In this case, both $\rad\sigma$ and $\rad\sigma^*$ are spanned by $\beta_{r+1},\dots,\beta_n$ over $\mathbb F_{q^2}$, and $C$ is congruent to the diagonal matrix with entries $d_1,\dots,d_r$ by the last proposition. It follows from Theorem \ref{equiv} that $C^*C^{-1}$ is conjugate to the diagonal matrix with entries $d_1^{q-1},\dots,d_r^{q-1}$. Reversing the same argument gives the converse.
\end{proof}

\begin{example}
Continue the discussion in Example \ref{2d} but assume that $\mathrm N(a)=\mathrm N(b)\ne0$. For $L(x)=ax^{q^2}+bx$ with $L^*(x)=b^{q^3}x^{q^2}+a^qx$, $\ker L=\ker L^*$ if and only if $\frac ba=\frac{a^q}{b^{q^3}}$. Suppose that $a^{q+1}=b^{q^3+1}$. Then $\rad\sigma_L=\rad\sigma_L^*$, and $\sigma_L$ has a nonzero coefficient matrix
\[\begin{pmatrix}\mathrm{Tr}(\beta L(\beta^q))&0\\0&0\end{pmatrix}\]
for some $\beta\in\mathbb F_{q^4}$. The problem is then reduced to the $1$-dimensional case, where the equivalence class is determined by $\mathrm{Tr}(\beta L(\beta^q))^{q-1}$. Note that $(a^qb^{-1})^{q+1}=1$ and
\[\begin{split}&\mathrel{\phantom{=}}\mathrm{Tr}\big(a\beta^{q^3+1}+b\beta^{q+1}\big)^q-a^qb^{-1}\mathrm{Tr}\big(a\beta^{q^3+1}+b\beta^{q+1}\big)\\&=\mathrm{Tr}\big(a^q\beta^{q+1}+b^q\beta^{q^2+q}-a^{q+1}b^{-1}\beta^{q^3+1}-a^q\beta^{q+1}\big)\\&=\mathrm{Tr}\big(b^{q^3}\beta^{q^3+1}-a^{q+1}b^{-1}\beta^{q^3+1}\big)\\&=\mathrm{Tr}\big(\big(b^{q^3+1}-a^{q+1}\big)b^{-1}\beta^{q^3+1}\big)\\&=0,\end{split}\]
so $\mathrm{Tr}(\beta L(\beta^q))^{q-1}=a^qb^{-1}$.
Therefore, if $a^{q+1}=b^{q^3+1}$, then the equivalence class of $\sigma_L$ uniquely corresponds to $a^qb^{-1}$. Otherwise, with respect to a basis $\beta_1,\beta_2$, where $\beta_1\in\rad\sigma_L$ and $\beta_2\in\rad\sigma_L^*$, the sesquilinear form has a nonzero coefficient matrix
\[\begin{pmatrix}0&0\\\mathrm{Tr}(\beta_2L(\beta_1^q))&0\end{pmatrix};\]
multiplying $\beta_1$ by some nonzero element of $\mathbb F_{q^2}$ gives a coefficient matrix
\[\begin{pmatrix}0&0\\1&0\end{pmatrix}.\]
There is only one equivalence class in this case. Finally, we have classified all nonzero sesquilinear forms over finite fields in the $2$-dimensional case.
\end{example}

\section{An associated quadratic forms}

For the sesquilinear form $\sigma_L$ of $\mathbb F_{q^{2n}}/\mathbb F_{q^2}$, taking its trace over $\mathbb F_q$ one gets a bilinear form of $\mathbb F_{q^{2n}}/\mathbb F_q$. In particular, $\sigma_L(x,x)^q+\sigma_L(x,x)$ induces a quadratic form over $\mathbb F_q$, denoted by $\rho_L$. In general, quadratic forms have coefficient matrices in some canonical forms (see Theorem 6.21 in \cite{lidl1997} for odd $q$ and Theorem 6.30 there for even $q$), which lead to their classification. Briefly, for a quadratic form over $\mathbb F_q$ with $q$ odd, we usually need to find its nondegenerate part and calculate the determinant; for even $q$, the problem is more complicated, involving symplectic bases for alternating forms. However, we will see that, for a quadratic form induced by a sesquilinear form, the classification depends only on the rank of its coefficient matrix.

Let $\psi$ be the canonical additive character of $\mathbb F_q$ and define
\[\mathcal S(L)=\sum_{u\in\mathbb F_{q^{2n}}}\psi(\rho_L(u)).\]
With $\rho_L$ written in its canonical form, it follows directly that this sum determines the equivalence class of the quadratic form (except when $q$ is even and the sum turns out to be zero). Our goal it to evaluate $\mathcal S(L)$ explicitly.

\begin{lemma}\label{diagonal}
Every $\lambda$-Hermitian form of $\mathbb F_{q^{2n}}/\mathbb F_{q^2}$ for some $\lambda\in\mathbb F_{q^2}$ has a diagonal coefficient matrix whose nonzero entries are $d_1,\dots,d_r$, where, for $1\le i\le r$, $d_i$ is an arbitrary element in $\mathbb F_{q^2}$ satisfying $d_i^{q-1}=\lambda$.
\end{lemma}
\begin{proof}
Given a $\lambda$-Hermitian form of $\mathbb F_{q^{2n}}/\mathbb F_{q^2}$, let $C$ be a coefficient matrix of its reduced form. It is clear that $\lambda^{q+1}=1$, $\rad\sigma=\rad\sigma^*$ and $C^*=\lambda C$. Considering $C^*C^{-1}$, we get the desired result by Corollary \ref{reduced}.
\end{proof}

\begin{theorem}
Let $r=n-\dim\ker(L^*+L)$. Then $\mathcal S(L)=(-1)^rq^{2n-r}$.
\end{theorem}
\begin{proof}
Observe that
\[\sum_{t\in\mathbb F_{q^2}}\psi(t^{q+1})=1+\sum_{t\in\mathbb F_{q^2}^*}\psi(t^{q+1})=1+(q+1)\sum_{w\in\mathbb F_q^*}\psi(w)=-q,\]
and
\[\begin{split}\mathcal S(L)&=\sum_{u\in\mathbb F_{q^{2n}}}\psi(\mathrm{Tr}(uL(u^q))^q+\mathrm{Tr}(uL(u^q)))\\&=\sum_{u\in\mathbb F_{q^{2n}}}\psi(\mathrm{Tr}(uL^*(u^q)+uL(u^q)))\\&=\sum_{u\in\mathbb F_{q^{2n}}}\psi(\mathrm{Tr}(u(L^*+L)(u^q)))\\&=\sum_{u_1,\dots,u_n\in\mathbb F_{q^2}}\psi(u_1^{q+1}+\dots+u_r^{q+1}),\end{split}\]
where the last equality follows from Lemma \ref{diagonal}. Then $\mathcal S(L)=(-q)^rq^{2(n-r)}=(-1)^rq^{2n-r}$.
\end{proof}

As a consequence, $\mathcal S(L)\ne0$ even if $q$ is even, and then the quadratic form $\rho_L$ is characterized by $\dim\ker(L^*+L)$. Furthermore, for the canonical additive character $\chi$ of $\mathbb F_{q^2}$, we also have
\[\mathcal S(L)=\sum_{u\in\mathbb F_{q^{2n}}}\chi(\sigma_L(u,u)).\]
This is essential for the next section in terms of equations associated with sesquilinear forms.

\section{Equations from sesquilinear forms}

Given the sesquilinear form $\sigma_L$ for a $q^2$-linear polynomial $L$ over $\mathbb F_{q^{2n}}$, consider $N_c$, the number of zeros of $\sigma_L(x,x)+c$ in $\mathbb F_{q^{2n}}$ for $c\in\mathbb F_{q^2}$. The evaluation of $\mathcal S(L)$ allows us to determine the number $N_c$ for each $c\in\mathbb F_{q^2}$. In this section, let $U$ be the subgroup of order $q+1$ of $\mathbb F_{q^2}^*$ and let $R(\upsilon)=n-\dim\ker(\upsilon L^*+L)$ for $\upsilon\in U$. 

\begin{theorem}\label{number}
We have
\[N_0=q^{2(n-1)}+(q-1)q^{2(n-1)}\sum_{\upsilon\in U}(-q)^{-R(\upsilon)}\]
and
\[N_c=q^{2(n-1)}+q^{2n-1}(-q)^{-R(-c^{1-q})}-q^{2(n-1)}\sum_{\upsilon\in U}(-q)^{-R(\upsilon)}\]
for $c\in\mathbb F_{q^2}^*$.
\end{theorem}
\begin{proof}
By the orthogonality relations for characters, the number of zeros of $\mathrm{Tr}(xL(x^q))+c$ in $\mathbb F_{q^{2n}}$ is
\[\begin{split}&\mathrel{\phantom{=}}q^{-2}\sum_{u\in\mathbb F_{q^{2n}}}\sum_{t\in\mathbb F_{q^2}}\chi(t(\mathrm{Tr}(uL(u^q))+c))\\&=q^{2(n-1)}+q^{-2}\sum_{t\in\mathbb F_{q^2}^*}\mathcal S(tL)\chi(ct)\\&=q^{2(n-1)}+q^{-2}\sum_{\upsilon\in U}\sum_{t^{q-1}=\upsilon}\mathcal S(tL)\chi(ct).\end{split}\]
Here, $\mathcal S(tL)=(-q)^{-R(\upsilon)}q^{2n}$ for $\upsilon=t^{q-1}$ as
\[R(\upsilon)=n-\dim\ker(t^{q-1}L^*+L)=n-\dim\ker(t^qL^*+tL),\]
and
\[\sum_{w\in\mathbb F_q^*}\chi(ctw)=\sum_{w\in\mathbb F_q^*}\psi((c^qt^q+ct)w),\]
where $c^qt^q+ct=0$ if and only if $c=0$ or $t^{q-1}=-c^{1-q}$. Hence, 
\[\sum_{t^{q-1}=\upsilon}\chi(ct)=\begin{cases}q-1&\text{if }\upsilon=-c^{1-q},\\-1&\text{otherwise,}\end{cases}\]
when $c\ne0$. The result follows from a straightforward calculation.
\end{proof}

\begin{corollary}
If $\sigma_L$ is $\lambda$-Hermitian for some $\lambda\in\mathbb F_{q^2}$, then, for $r=n-\dim\ker L$,
\[N_c=\begin{cases}q^{2n-1}+(-1)^r(q-1)q^{2n-r-1}&\text{if }c=0,\\q^{2n-1}+(-1)^{r-1}q^{2n-r-1}&\text{if }c^{q-1}=\lambda,\\0&\text{otherwise.}\end{cases}\]
\end{corollary}
\begin{proof}
Note that $\sum_{\upsilon\in U}(-q)^{-R(\upsilon)}=1+q(-q)^{-r}=1+(-1)^rq^{1-r}$ as
\[\dim\ker(\upsilon L^*+L)=\dim\ker((\upsilon\lambda+1)L)=\begin{cases}n&\text{if }\upsilon\lambda+1=0,\\n-r&\text{otherwise.}\end{cases}\]
Thus,
\[N_0=q^{2(n-1)}+(q-1)q^{2(n-1)}(1+(-1)^rq^{1-r})=q^{2n-1}+(-1)^r(q-1)q^{2n-r-1}.\]
For $c\in\mathbb F_{q^2}^*$, if $c^{q-1}=\lambda$, then $R(-c^{1-q})=0$ and 
\[N_c=q^{2(n-1)}+q^{2n-1}-q^{2(n-1)}(1+(-1)^rq^{1-r})=q^{2n-1}+(-1)^{r-1}q^{2n-r-1};\]
otherwise, $R(-c^{1-q})=r$ and
\[N_c=q^{2(n-1)}+q^{2n-1}(-q)^{-r}-q^{2(n-1)}(1+(-1)^rq^{1-r})=0,\]
as desired.
\end{proof}

\begin{theorem}\label{bound}
Let $L$ have degree $q^{2m}$ for some nonnegative integer $m$. Then $|N_c-q^{2(n-1)}|\le(q^2-1)q^{n+2m-1}$, and the equality holds if and only if $c=0$ and $\dim\ker(\upsilon L^*+L)=2m+1$ for every $\upsilon\in U$. In this case, $N_0-q^{2(n-1)}=(-1)^{n-1}(q^2-1)q^{n+2m-1}$.
\end{theorem}
\begin{proof}
Recall that if $L(x)=\sum_{i=0}^ma_ix^{q^{2i}}$, then
\[\upsilon L^*(x)+L(x)=\upsilon\sum_{i=0}^m(a_i^qx)^{q^{2(n-i-1)}}+\sum_{i=0}^ma_ix^{q^i}\]
for $\upsilon\in U$, and
\[(\upsilon L^*(x)+L(x))^{q^{2(m+1)}}\equiv\upsilon\sum_{i=0}^m(a_i^qx)^{q^{2(m-i)}}+\sum_{i=0}^ma_i^{q^{2(m+1)}}x^{q^{2(m+i+1)}}\pmod{x^{q^{2n}}-x}.\]
This implies $\dim\ker(\upsilon L^*+L)\le2m+1$. Furthermore, as in the proof of Theorem \ref{number},
\[q^2\big(N_c-q^{2(n-1)}\big)=\sum_{t\in\mathbb F_{q^2}^*}\mathcal S(tL)\chi(ct),\]
and
\[\Bigg|\sum_{t\in\mathbb F_{q^2}^*}\mathcal S(tL)\chi(ct)\Bigg|\le\sum_{t\in\mathbb F_{q^2}^*}|\mathcal S(tL)\chi(ct)|=\sum_{t\in\mathbb F_{q^2}^*}|q^{2n-R(t^{q-1})}|\le\sum_{t\in\mathbb F_{q^2}^*}q^{n+2m+1}.\]
Clearly, $|N_c-q^{2(n-1)}|\le(q^2-1)q^{n+2m-1}$ and the equality holds if and only if $R(t^{q-1})=n-2m-1$ for all $t\in\mathbb F_{q^2}^*$, and the complex argument of $\mathcal S(tL)\chi(ct)$ is a constant independent of $t$; i.e., $\dim\ker(\upsilon L^*+L)=2m+1$ for all $\upsilon\in U$ and $c=0$. Here, the sign of $N_0-q^{2(n-1)}$ is $(-1)^{n-2m-1}=(-1)^{n-1}$.
\end{proof}

Now we investigate the Artin-Schreier curve defined by the polynomial $y^{q^2}-y-xL(x^q)$ in $\mathbb F_{q^{2n}}[x,y]$. When $L$ has degree $q^{2m}$, the curve has genus $\frac12(q^2-1)q^{2m+1}$ (see, e.g., \cite[Proposition 3.7.10]{stichtenoth2009}), and the number $N$ of its rational points over $\mathbb F_{q^{2n}}$ satisfies
\[q^{2n}+1-(q^2-1)q^{n+2m+1}\le N\le q^{2n}+1+(q^2-1)q^{n+2m+1},\]
by the Hasse-Weil Theorem. The curve is called maximal if $N$ achieves the upper bound, and called minimal if it achieves the lower bound. Since the curve has exactly one point at infinity, and every zero of $\mathrm{Tr}(xL(x^q))$ in $\mathbb F_{q^{2n}}$ corresponds to $q^2$ rational points of it, the preceding theorem indeed characterizes all those maximal or minimal Artin-Schreier curves, and whether it is maximal or minimal depends only on $(-1)^{n-1}$. In what follows, we give some explicit instances.

\begin{proposition}
Let $L(x)=ax^{q^{2m}}$ for $a\in\mathbb F_{q^{2n}}^*$ and a nonnegative integer $m$. Let $d=\gcd(2m+1,n)$, $\alpha=\big(-a^{q^{2m+1}-1}\big)^\frac{q^{2n}-1}{q^{2d}-1}$ and $Q=q+1-\gcd(n/d,q+1)(q^d+1)$. If $\alpha^\frac{q+1}{\gcd(n/d,q+1)}=1$, then
\[N_c-q^{2(n-1)}=\begin{cases}(-1)^n(q-1)q^{n-2}Q&\text{if }c=0,\\(-1)^{n-1}q^{n+d-1}+(-1)^{n-1}q^{n-2}Q&\text{if }(-c^{q-1})^\frac nd=\alpha,\\(-1)^nq^{n-1}+(-1)^{n-1}q^{n-2}Q&\text{otherwise.}\end{cases}\]
If $\alpha^\frac{q+1}{\gcd(n/d,q+1)}\ne1$, then
\[N_c-q^{2(n-1)}=\begin{cases}(-1)^n(q^2-1)q^{n-2}&\text{if }c=0,\\(-1)^{n-1}q^{n-2}&\text{otherwise.}\end{cases}\]
In particular, $|N_0-q^{2(n-1)}|=(q^2-1)q^{n+2m-1}$ if and only if $(2m+1)(q+1)$ divides $n$ and $\alpha=1$.
\end{proposition}
\begin{proof}
For $L^*(x)=a^{q^{2(n-m)-1}}x^{q^{2(n-m-1)}}$ and $\upsilon\in U$, we have
\[(\upsilon L^*(x)+L(x))^{q^{2(m+1)}}=\upsilon a^qx^{q^{2n}}+a^{q^{2(m+1)}}x^{q^{2(2m+1)}}.\]
Note that $\gcd(q^{2(2m+1)}-1,q^{2n}-1)=q^{2d}-1$, and thus
\[\dim\ker(\upsilon L^*+L)=\begin{cases}d&\text{if }\big(-\upsilon^qa^{1-q^{2m+1}}\big)^\frac{q^{2n}-1}{q^{2d}-1}=1,\\0&\text{otherwise,}\end{cases}\]
where $\big(-\upsilon^qa^{1-q^{2m+1}}\big)^\frac{q^{2n}-1}{q^{2d}-1}=\upsilon^{q\frac nd}\alpha^{-1}$. If $\upsilon$ has order $q+1$, then $\upsilon^{q\frac nd}$ has order $\frac{q+1}{\gcd(n/d,q+1)}$. Thus, there exists some $\upsilon\in U$ such that $\dim\ker(\upsilon L^*+L)=d$ if and only if $\alpha^\frac{q+1}{\gcd(n/d,q+1)}=1$; in this case, the number of such $\upsilon$ is $\gcd(n/d,q+1)$. This, combined with Theorem \ref{bound}, proves the statement on the upper bound of $|N_0-q^{2(n-1)}|$.

Suppose $\alpha^\frac{q+1}{\gcd(n/d,q+1)}=1$. Then
\[\begin{split}\sum_{\upsilon\in U}(-q)^{-R(\upsilon)}&=\gcd(n/d,q+1)(-q)^{d-n}+(q+1-\gcd(n/d,q+1))(-q)^{-n}\\&=(-q)^{-n}(q+1-\gcd(n/d,q+1)(q^d+1))\\&=(-q)^{-n}Q.\end{split}\]
Accordingly,
\[N_0=q^{2(n-1)}+(q-1)q^{2(n-1)}(-q)^{-n}Q=q^{2(n-1)}+(-1)^n(q-1)q^{n-2}Q,\]
and
\[N_c=q^{2(n-1)}+q^{2n-1}(-q)^{-R(-c^{1-q})}-q^{2(n-1)}(-q)^{-n}Q\]
for $c\in\mathbb F_{q^2}^*$, where
\[R(-c^{1-q})=\begin{cases}n-d&\text{if }(-c^{q-1})^\frac nd=\alpha,\\n&\text{otherwise,}\end{cases}\]
as desired.

If $\alpha^\frac{q+1}{\gcd(n/d,q+1)}\ne1$, then $R(\upsilon)=n$ for all $\upsilon\in U$, and thus
\[N_0=q^{2(n-1)}+(q-1)q^{2(n-1)}(q+1)(-q)^{-n}=q^{2(n-1)}+(-1)^n(q^2-1)q^{n-2},\]
with
\[N_c=q^{2(n-1)}+q^{2n-1}(-q)^{-n}-q^{2(n-1)}(q+1)(-q)^{-n}=q^{2(n-1)}+(-1)^{n-1}q^{n-2}\]
for $c\in\mathbb F_{q^2}^*$.
\end{proof} 

\begin{proposition}
Let $L(x)=\gamma^{q^{2m+1}+1}\delta x^{q^{2m}}-\gamma^{q^{2l+1}+1}\delta x^{q^{2l}}$ for $\gamma\in\mathbb F_{q^{2n}}^*$, $\delta\in\mathbb F_{q^{2(m-l)}}^*$, and integers $m,l$ with $0\le l<m$. Let $d=\gcd(m-l,m+l+1)$ and $s=\frac{q^{2(m-l)}-1}{q^d+1}\frac{q+1}{\gcd((m-l)/d,q+1)}$. Suppose that $n=k\lcm(m-l,(m+l+1)(q+1))$ for some positive integer $k$, and $\delta^{sk}=1$ with either $\delta^s\ne1$ or $\gcd(k,q)\ne1$. Then $N_0-q^{2(n-1)}=(-1)^{n-1}(q^2-1)q^{n+2m-1}$.
\end{proposition}
\begin{proof}
It suffices to show that $\dim\ker(\upsilon L^*+L)=2m+1$ for every $\upsilon\in U$. Let $\alpha=-\gamma^{q^{2(m+l+1)}-1}\delta^{q^{2l+1}-1}$, so that
\[\begin{split}\alpha L(x)&=-\gamma^{q^{2(m+l+1)}+q^{2m+1}}\delta^{q^{2l+1}}x^{q^{2m}}+\gamma^{q^{2(m+l+1)}+q^{2l+1}}\delta^{q^{2l+1}}x^{q^{2l}}\\&=-\big(\gamma^{q^{2(l+1)}+q}\delta^qx\big)^{q^{2m}}+\big(\gamma^{q^{2(m+1)}+q}\delta^qx\big)^{q^{2l}},\end{split}\]
while
\[L^*(x)=(\gamma^{q^{2(m+1)}+q}\delta^qx\big)^{q^{2(n-m-1)}}-\big(\gamma^{q^{2(l+1)}+q}\delta^qx\big)^{q^{2(n-l-1)}}.\]
Then $\upsilon L^*(x)+L(x)=L_0(L(x))$ for $\upsilon\in U$, where $L_0(x)=\upsilon(\alpha x)^{q^{2(n-m-l-1)}}+x$. Moreover, $\dim\ker L=m-l$ and $\im L=\gamma\ker\mathrm{Tr}_{m-l}$, where $\mathrm{Tr}_{m-l}$ denotes the trace map of $\mathbb F_{q^{2n}}/\mathbb F_{q^{2(m-l)}}$, since
\[L(x)=\gamma\delta\big((\gamma^q x)^{q^{2m}}-(\gamma^qx)^{q^{2l}}\big).\]
Let $e=\lcm(m-l,m+l+1)$ and $\theta=\big(\upsilon\delta^{q^{2l+1}-1}\big)^\frac{q^{2e}-1}{q^{2(m+l+1)}-1}$, for which we claim that
\[\theta^\frac ne=1\quad\text{and either}\quad\theta\ne1\quad\text{or}\quad\gcd(n/e,q)\ne1.\]
Then
\[(-\upsilon\alpha)^\frac{q^{2n}-1}{q^{2(m+l+1)}-1}=(\upsilon\delta^{q^{2l+1}-1})^{\frac{q^{2e}-1}{q^{2(m+l+1)}-1}\frac{q^{2n}-1}{q^{2e}-1}}=\theta^\frac ne=1,\]
and thus $\upsilon\alpha=-\beta^{1-q^{2(m+l+1)}}$ for some $\beta\in\mathbb F_{q^{2n}}$. It follows that
\[L_0(x)=-\beta^{q^{2(n-m-l-1)}-1}x^{q^{2(n-m-l-1)}}+x=\beta^{-1}\big(\beta x-(\beta x)^{q^{2(n-m-l-1)}}\big),\]
with $\ker L_0=\beta^{-1}\mathbb F_{q^{2(m+l+1)}}$. Note that $(\beta\gamma)^{1-q^{2(m+l+1)}}=\upsilon\delta^{q^{2l+1}-1}$ and, for the trace map $\mathrm{Tr}_e$ of $\mathbb F_{q^{2n}}/\mathbb F_{q^{2e}}$,
\[\begin{split}\beta\gamma\mathrm{Tr}_e((\beta\gamma)^{-1})=\sum_{i=0}^{\frac ne-1}(\beta\gamma)^{{1-q^{2(m+l+1)}}\frac{q^{2e}-1}{q^{2(m+l+1)}-1}\frac{q^{2ei}-1}{q^{2e}-1}}=\sum_{i=0}^{\frac ne-1}\theta^i=0\end{split}\]
by the above claim. Now $\mathrm{Tr}_{m-l}((\beta\gamma)^{-1}x)$ vanishes on $\mathbb F_{q^{2(m+l+1)}}$ since $\mathrm{Tr}_e((\beta\gamma)^{-1}u)=\mathrm{Tr}_e((\beta\gamma)^{-1})u=0$ for all $u\in\mathbb F_{q^{2(m+l+1)}}$. This implies
\[\ker L_0=\beta^{-1}\mathbb F_{q^{2(m+l+1)}}\subseteq\gamma\ker\mathrm{Tr}_{m-l}=\im L,\]
so
\[\dim\ker(\upsilon L^*+L)=\dim\ker L+\dim(\ker L_0\cap\im L)=m-l+m+l+1=2m+1\]
for arbitrary $\upsilon\in U$.

It remains to prove the claim on $\theta$. For $d=\gcd(m-l,m+l+1)$, we have
\[\frac ne=k\frac{\lcm(m-l,(m+l+1)(q+1))}{\lcm(m-l,m+l+1)}=k\frac{q+1}{\gcd((m-l)/d,q+1)},\]
with $\gcd(n/e,q)=\gcd(k,q)$, and
\[\frac{q^{2e}-1}{q^{2(m+l+1)}-1}=\sum_{i=0}^{\frac e{m+l+1}-1}q^{2(m+l+1)i}\equiv\sum_{i=0}^{\frac{m-l}d-1}q^{2di}\equiv\frac{q^{2(m-l)}-1}{q^{2d}-1}\pmod{q^{2(m-l)}-1},\]
because $m+l+1$ and $d$ generate the same subgroup of order $\frac{m-l}d$ in $\mathbb Z/(m-l)\mathbb Z$. Then
\[\gcd\Big(\frac{q^{2e}-1}{q^{2(m+l+1)}-1},q^{2(m-l)}-1\Big)=\frac{q^{2(m-l)}-1}{q^{2d}-1},\]
and
\[\gcd\Big((q^{2l+1}-1)\frac{q^{2e}-1}{q^{2(m+l+1)}-1},q^{2(m-l)}-1\Big)=\frac{q^{2(m-l)}-1}{q^{2d}-1}\gcd(q^{2l+1}-1,q^{2d}-1),\]
where $\gcd(q^{2l+1}-1,q^{2d}-1)=q^d-1$ as $\gcd(2l+1,2d)=\gcd(2l+1,d)=d$. A similar argument gives
\[\begin{split}&\mathrel{\phantom{=}}\gcd\Big((q^{2l+1}-1)\frac{q^{2e}-1}{q^{2(m+l+1)}-1}\frac{q+1}{\gcd(e/(m+l+1),q+1)},q^{2(m-l)}-1\Big)\\&=\frac{q^{2(m-l)}-1}{q^d+1}\gcd\Big(\frac{q+1}{\gcd((m-l)/d,q+1)},q^d+1\Big)\\&=s,\end{split}\]
and
\[\begin{split}&\mathrel{\phantom{=}}\gcd\Big((q^{2l+1}-1)\frac{q^{2e}-1}{q^{2(m+l+1)}-1}\frac ne,q^{2(m-l)}-1\Big)\\&=\frac{q^{2(m-l)}-1}{q^d+1}\gcd(n/e,q^d+1)\\&=\gcd(sk,q^{2(m-l)}-1),\end{split}\]
Therefore,
\[\theta^\frac ne=\big(\upsilon\delta^{q^{2l+1}-1}\big)^{\frac{q^{2e}-1}{q^{2(m+l+1)}-1}\frac ne}=\upsilon^\frac n{m+l+1}\delta^{(q^{2l+1}-1)\frac{q^{2e}-1}{q^{2(m+l+1)}-1}\frac ne}=1,\]
since $q+1$ divides $\frac n{m+l+1}$ and $\delta^{sk}=1$. If $\delta^s\ne1$, then $\theta\ne1$; otherwise,
\[\delta^{(q^{2l+1}-1)\frac{q^{2e}-1}{q^{2(m+l+1)}-1}\frac{q+1}{\gcd(e/(m+l+1),q+1)}}=\big(\upsilon^{-\frac e{m+l+1}}\big)^\frac{q+1}{\gcd(e/(m+l+1),q+1)}=1,\]
a contradiction. The proof is now complete.
\end{proof}

Finally, we have completely determined the number of zeros of $\mathrm{Tr}(xL(x^q))$ in $\mathbb F_{q^{2n}}$ when $L$ is a monomial, and derived some maximal or minimal Artin-Schreier curves defined by binomials.

\bibliographystyle{abbrv}
\bibliography{references}

\begin{thebibliography}{10}

\bibitem{anbar2014quadratic}
N.~Anbar and W.~Meidl.
\newblock Quadratic functions and maximal Artin--Schreier curves.
\newblock {\em Finite Fields and Their Applications}, 30:49--71, 2014.

\bibitem{bartoli2021explicit}
D.~Bartoli, L.~Quoos, Z.~Sayg{\i}, and E.~S. Y{\i}lmaz.
\newblock Explicit maximal and minimal curves of Artin--Schreier type from
  quadratic forms.
\newblock {\em Applicable Algebra in Engineering, Communication and Computing},
  32(4):507--520, 2021.

\bibitem{coulter2002number}
R.~S. Coulter.
\newblock The number of rational points of a class of Artin--Schreier curves.
\newblock {\em Finite Fields and Their Applications}, 8(4):397--413, 2002.

\bibitem{fitzgerald2017invariants}
R.~W. Fitzgerald and Y.~Kottegoda.
\newblock Invariants of trace forms in odd characteristic and authentication
  codes.
\newblock {\em Finite Fields and Their Applications}, 48:1--9, 2017.

\bibitem{klapper1997cross}
A.~Klapper.
\newblock Cross-correlations of quadratic form sequences in odd characteristic.
\newblock {\em Designs, Codes and Cryptography}, 11(3):289--305, 1997.

\bibitem{lidl1997}
R.~Lidl and H.~Niederreiter.
\newblock {\em Finite fields}.
\newblock Cambridge University Press, 1997.

\bibitem{oliveira2023artin}
D.~Oliveira and F.~B. Mart{\'\i}nez.
\newblock Artin-Schreier curves given by $\mathbb F_q$-linearized polynomials.
\newblock {\em Discrete Mathematics}, 346(12):113630, 2023.

\bibitem{ozbudak2016explicit}
F.~{\"O}zbudak and Z.~Sayg{\i}.
\newblock Explicit maximal and minimal curves over finite fields of odd
  characteristics.
\newblock {\em Finite Fields and Their Applications}, 42:81--92, 2016.

\bibitem{stichtenoth2009}
H.~Stichtenoth.
\newblock {\em Algebraic function fields and codes}.
\newblock Springer Science \& Business Media, 2009.

\bibitem{wall1963conjugacy}
G.~Wall.
\newblock On the conjugacy classes in the unitary, symplectic and orthogonal
  groups.
\newblock {\em Journal of the Australian Mathematical Society}, 3(1):1--62,
  1963.

\end{thebibliography}

\end{document}